\documentclass[11pt]{amsart}

\usepackage[marginratio=1:1,height=600pt,width=400pt,tmargin=110pt]{geometry}

\usepackage{amsmath, amsfonts, amssymb, amscd, bezier, amstext, amsthm}
\usepackage{graphicx}
\usepackage{microtype} 
\usepackage{hyperref}

\DeclareMathOperator{\tr}{tr}
\DeclareMathOperator{\Div}{div}
\DeclareMathOperator{\Ric}{Ric}

\DeclareMathOperator{\reg}{reg}
\DeclareMathOperator{\sing}{sing}

\newcommand{\R}{\mathbb{R}}
\newcommand{\N}{\mathbb{N}}
\newcommand{\e}{\varepsilon}

\theoremstyle{plain}

\newcounter{alpha}

\newtheorem{prop_}{Proposition}[section]
\newtheorem{prop}[prop_]{Proposition}

\newtheorem{lema}[prop_]{Lemma}
\newtheorem*{lema0}{Lemma}

\newtheorem{teo}[prop_]{Theorem}
\newtheorem*{teo0}{Theorem}
\newtheorem{theorema}[alpha]{Theorem}

\newtheorem*{obs}{Remark}

\title[\resizebox{5.2in}{!}{The second inner variation of energy and the Morse index of limit interfaces}]{The second inner variation of energy and the Morse index of limit interfaces}
\author{Pedro Gaspar} 
\address{Instituto de Matem\'atica Pura e Aplicada (IMPA) \\ Estrada Dona Castorina 110 \\ 22460-320 Rio de Janeiro \\ Brazil}
\email{phgms@impa.br}

\thanks{The author was partly supported by NSF grant DMS-1311795}

\begin{document}

\begin{abstract} In this article we study the second variation of the energy functional associated to the Allen-Cahn equation on closed manifolds. Extending well known analogies between the gradient theory of phase transitions and the theory of minimal hypersurfaces, we prove the upper semicontinuity of the eigenvalues of the stability operator and consequently obtain upper bounds for the Morse index of limit interfaces which arise from solutions with bounded energy and index without assuming any multiplicity or orientability condition on these hypersurfaces. This extends some recent results of N. Le \cite{LeInner,LeEigen} and F. Hiesmayr \cite{Fritz}.
\end{abstract}

\maketitle

\section{Introduction}

In this article we are interested in understanding the limit behavior of solutions to the elliptic \emph{Allen-Cahn equation} on a closed Riemannian manifold $M^n$, $n\geq 3$, namely $u:M \to \R$ such that
	\begin{equation} \label{eq:ac}
		-\e \Delta u + W'(u)/\e=0, 
	\end{equation}
as $\e$ goes to zero. Here $W$ is a double-well potential, such as $W(u)=(1-u^2)^2/4$. This equation and its parabolic counterpart arise in the gradient theory of phase transition phenomena within the van der Waals-Cahn-Hilliard theory \cite{AllenCahn}. The solutions to these equations are known to be related to critical points of the area functional. This connection was studied in the context of $\Gamma$-convergence by Modica-Mortola \cite{ModicaMortola} (see also \cite{KohnSternberg}) where the convergence of minimizers of the associated energy functional
	\begin{equation}
		E_\e(u) = \int_M \e\frac{|\nabla u|^2}{2} + \frac{W(u)}{\e}, \quad u \in H^1(M),
	\end{equation}
to critical points of the area functional is studied. For the corresponding problem with the constraint $\int_M u = c$ we refer to Modica \cite{Modica} and Sternberg \cite{Sternberg}. Since then strong parallels between these objects have been drawn, see e.g. the surveys \cite{Pacard, Savin} and the references therein.

For more general variational solutions of the Allen-Cahn equation refined results about the limit behavior were carried out by Hutchinson-Tonegawa \cite{HutchinsonTonegawa}, Tonegawa \cite{Tonegawa}, Tonegawa-Wickramasekera \cite{TonegawaWickramasekera} and Guaraco \cite{Guaraco}. Roughly speaking they proved

\begin{teo0} \label{thm:conv0}
Let $M$ be a closed Riemannian manifold of dimension $n\geq 3$, and let $\{u_{\e_k}\}$ be a sequence of solutions to \eqref{eq:ac} in M with $\e=\e_k \downarrow 0$. Assume that the sequences $\sup_M|u_{\e_k}|$, $E_{\e_k}(u_{\e_k})$ and  $m(u_{\e_k})$ are bounded, where $m(u_\e)$ denotes the Morse index of $u_\e$ as a critical point of $E_\e$. Then as $\e_k \downarrow 0$ its level sets accumulate around a minimal hypersurface $\Gamma \subset M$ which is smooth and embedded outside a singular set $\sing \Gamma$ of Hausdorff dimension at most $n-8$. Moreover there are positive integers $m_1,\ldots, m_N$ such that
	\[\lim_k E_{\e_k}(u_{\e_k}) = 2\sigma \sum_{j=1}^N m_j \mathcal{H}^{n-1}(\Gamma_j),\]
where $\Gamma_1,\ldots, \Gamma_N$ are the connected components of $\Gamma$, and $\sigma = \int_{-1}^1 \sqrt{W/2}$.
\end{teo0}

The minimal hypersurface $\Gamma$ is called the \emph{limit interface} associated to the corresponding solutions, and the positive integers $m_j$ are the \emph{multiplicities} of $\Gamma_j$. For a more refined notion of convergence see Theorem \ref{thm:conv1} below. 

\begin{obs}
For $n=2$ similar conclusions hold, except the regularity of the limit interface. In this case, the limit varifold is supported in an union $\Gamma$ of geodesic arcs with at most $p=\limsup_k m(u_{\e_k})$ junction points, according to \cite{Tonegawa}. It was proved recently by C. Mantoulidis \cite{Mantoulidis1} that if $p=1$ then $\Gamma$ is an immersed geodesic and the possible junction point is a transverse intersection.
\end{obs}

Using this convergence result along with min-max techniques for semilinear PDEs, Guaraco was able to provide an alternative proof of the celebrated result of Almgren-Pitts and Schoen-Simon about the existence of closed minimal hypersurfaces in closed Riemannian manifolds. This phase transitions approach simplifies considerably the variational argument of Almgren-Pitts to prove the existence of a stationary limit but it relies on the regularity theory of Wickramasekera \cite{Wickramasekera}, which is a sharpening of the classical Schoen-Simon compactness theory for stable minimal hypersurfaces. One is led then to the problem of describing the limit interfaces which arise from this strategy. For some results along these lines see, for instance, \cite{WangWei,Mantoulidis1}, for finite index solutions on surfaces, and \cite{GasparGuaraco} for least area limit interfaces.

One may expect, for example, to estimate the Morse index of limit interfaces in terms of the stability index of the solutions. In \cite{LeInner} N. Le obtained such an estimate in an Euclidean domain assuming either that the limit interface has multiplicity one, or that it is connected and the solutions satisfy an additional hypothesis. More precisely if $\{u_\e\}$ are critical points for the energy functional $E_\e$, he calculates the second \emph{inner} variation of $E_\e$ at $u_\e$ -- which is given by precomposing the solution with the flow generated by a compactly supported smooth vector field -- in terms of the energy density, i.e. the measure $de_\e = (\e|\nabla u|^2/2 + W(u)/\e)d\mathcal{H}^n$, and the first derivatives of $u_\e$. Le uses then a convergence result about the vector measures $\e\nabla u_\e \otimes \nabla u_\e \, d\mathcal{H}^n$ to prove that the second inner variation of $E_\e$ at $u_\e$ converges to the second variation of the corresponding limit interface, plus an error term. This strategy gives the following

\begin{teo0}[\cite{LeInner}] Assume $\{u_{\e_k}\}$ satisfy the conditions of the Theorem above. Assume also that either
	\begin{itemize}
		\item[(A)] $m_1=\ldots=m_N=1$, or
		\item[(B)] $N=1$ and the \emph{equipartition of energy} holds i.e.
			\[\limsup_k \int_M \left| \frac{\e_k|\nabla u_{\e_k}|^2}{2} - \frac{W''(u_{\e_k})}{\e_k} \right|\,d\mathcal{H}^n =0.\]
	\end{itemize}
Then the regular part of the limit interface has Morse index at most $p$.
\end{teo0}

\begin{obs}
Le's result is more general in the sense one need not to assume the solutions have uniformly bounded index provided the limit interface is as regular as guaranteed by Guaraco-Hutchinson-Tonegawa-Wickramasekera theorem. Furthermore, assuming also that the solutions are $L^1$-local minimizers for $E_\e$ Le proves in \cite{LeEigen} the upper semicontinuity of the eigenvalues of the corresponding Jacobi operators, see \cite[Corollary 1.1]{LeEigen} and Theorem \ref{thm:a} below.
\end{obs}

A similar result was obtained recently by F. Hiesmayr \cite{Fritz} under different conditions. Namely he assumes the limit interface is two-sided and obtains the same conclusion about its Morse index and the eigenvalues of the stability operator. Hiesmayr's proof follows a completely different strategy and relies on an inductive argument which is based on Tonegawa's work \cite{Tonegawa} and on $L^2$ bounds for the second fundamental form of the level sets of the solutions to the Allen-Cahn equation. These $L^2$ bounds can be used to prove curvature estimates for solutions to \eqref{eq:ac} in surfaces which resemble the estimates on the second fundamental form of stable minimal hypersurfaces obtained by Schoen and Simon \cite{SchoenSimon}, see \cite{WangWei} and \cite{Mantoulidis1}. 

In this article we prove

\begin{theorema} \label{thm:a}
Let $M$ be closed Riemannian manifold of dimension $n\geq 3$, and $\{u_{\e_k}\}$ a sequence of solutions to \eqref{eq:ac} with $\e=\e_k \downarrow 0$. Assume that there are positive constants $c_0$ and $E_0$, and a nonnegative integer $p$ such that
	\[\limsup_k \sup_M|u_{\e_k}| \leq c_0, \quad \limsup_k E_{\e_k}(u_{\e_k}) \leq E_0 \quad \mbox{and} \quad \limsup_k m(u_{\e_k}) \leq p. \] 
Then after perhaps passing to a subsequence, for all open subsets $U \subset \! \subset M \setminus \sing \Gamma$, the eigenvalues $\{\lambda_\ell^{\e_k}(U)\}_{\ell}$ of the linearized Allen-Cahn operator at $u_{\e_\ell}|_U$, that is $L_k=-\e_k\Delta + W''(u_{\e_k})/\e_k$, and the eigenvalues $\{\lambda_\ell(U)\}_\ell$ of the Jacobi operator of the corresponding limit interface on $U$, i.e. $LX=-\Delta^\perp X + (\Ric(X,X)+|A_\Gamma|^2|X|^2)$ acting on compactly supported normal vector fields $X$ on $(\Gamma \setminus \sing \Gamma) \cap U$, satisfy
	\[ \limsup_{k} \frac{\lambda_\ell^{\e_k}(U)}{\e_k} \leq \lambda_\ell(U) \]
for all $\ell$. In particular, the regular part of $\Gamma$ has Morse index at most $p$.
\end{theorema}

This theorem extends Le's and Hiesmayr's results in the sense that it is not necessary to assume any additional hypothesis on the limit interface. These bounds can be compared to the recent Morse index bounds obtained by Marques and Neves \cite{MarquesNevesIndex} in the context of Almgren-Pitts min-max theory for minimal hypersurfaces. 

The existence of solutions to the Allen-Cahn equation satisfying the hypothesis of Theorem \ref{thm:a} can be proven via topological methods, as in \cite{Smith}, and also by variational techniques, such as the multiparameter min-max construction of Guaraco and the author in \cite{GasparGuaraco}. In this latter strategy, our result shows that one may bound the Morse index of the corresponding limit interfaces from above by the number of parameters employed in the construction of solutions to \eqref{eq:ac}. This is similar in spirit to \cite{MarquesNevesIndex}. Moreover, this min-max construction for phase transitions was inspired by Marques-Neves' proof of Yau's conjecture on the existence of infinitely many minimal hypersurfaces, in the case of closed manifolds of positive Ricci curvature, via Almgren-Pitts min-max theory \cite{MarquesNevesInfinite}. Hence the present work can be seen as another instance of the analogy between phase transitions and minimal hypersurfaces within the framework of min-max techniques. 

Our strategy to prove Theorem \ref{thm:a} follows essentially the ideas of \cite{LeInner,LeEigen} by replacing the convergence of $\e\nabla u_\e \otimes \nabla u_\e \, d\mathcal{H}^n$ by the \emph{varifold convergence} of the solutions, and by observing -- as pointed by Hiesmayr \cite{Fritz} -- that the eigenvalues for the Jacobi operator of the limit interface have a variational characterization which allows for multiplicities as weights of the different components of $\Gamma$.

\subsection*{Outline of the paper.} In Section \ref{sec:conv} we briefly recall the convergence results of solutions of \eqref{eq:ac} to limit interfaces in the sense of varifolds and explain how can we use this convergence to describe the limit of the terms which arise in the second inner variation formula. In Section \ref{sec:2var} we re-derive Le's second inner variation formula for the Allen-Cahn energy in Riemannian manifolds, and use the results of Section \ref{sec:conv} to compare the second variations of energy and area of the limit interface. In Section \ref{sec:thma} we prove our main result. Finally the \hyperref[app:ext]{Appendix} contains an extension result about normal vectors fields which allows us to compare the aforementioned variations. 

\subsection*{Acknowledgements.} This work is partially based on the my Ph.D. thesis at IMPA -- Brazil. I would like to thank my advisor Fernando Cod\'a Marques for his constant encouragement and support. This work was carried out while visiting the Mathematics Department of Princeton University during 2017--18. I'm grateful to this institution for its kind hospitality and its support.

\section{Convergence of diffuse measures and varifolds} \label{sec:conv}

A crucial step in Le's proof of the Morse index bounds for the limit interface is the convergence of the vector measures $\lambda_\e = \e \nabla u_\e \otimes \nabla u_\e \,d\mathcal{H}^n$ to a $(n-1)$-dimensional measure concentrated along the limit interface $\Gamma$. Using the theory of Reshetnyak \cite{Reshetnyak} Le obtains this convergence assuming the regularity of $\Gamma$ and either that it has multiplicity 1, or that it is connected and the solutions $\{u_\e\}$ satisfy the equipartition of energy. The latter condition appeared on Hutchinson and Tonegawa's work and it holds for fairly general families of solutions to the Allen-Cahn equation, namely uniformly bounded solutions with bounded energy. We show in this section that for such solutions the convergence of the vector measures $\lambda_\e$ is a consequence of the \emph{varifold} convergence obtained in \cite{HutchinsonTonegawa,Tonegawa,TonegawaWickramasekera,Guaraco},

We begin by recalling these results. For definitions, notation and basic properties regarding the notion of varifolds we refer the reader to \cite{Simon} or \cite[Section 2]{HutchinsonTonegawa}. This notion was employed first in the context of phase transitions for the parabolic setting by Ilmanen in \cite{Ilmanen} in relation to the mean curvature flow. 

From now on we will make the follow assumption:

\begin{itemize}
\item [A.\ ] The function $W \in C^3(\R)$ is a nonnegative \emph{double-well potential}, that is, $W\geq 0$ and it has exactly three critical points, two of which are non-degenerated minima at $\pm1$, with ${W(\pm1)=0}$ and $W''(\pm1)>0$, and the third is a local maximum point in $(-1,1)$.
\end{itemize}

Furthermore we denote $\sigma=\int_{-1}^1\sqrt{W/2}$. Given a solution $u_\e \in C^1(M)$ to \eqref{eq:ac} we can define an \emph{associated $(n-1)$-varifold} $V^\e$ in $M$ by
	\[V^\e(\phi) = \int_{M \cap \{\nabla u_\e \neq 0\}}\e\frac{|\nabla u_\e(x)|^2}{2}\phi(x,T_x\{u_\e=u_\e(x)\})\,d\mathcal{H}^{n-1}(x)\]
for all compactly supported continuous functions $\phi:G_{n-1}(M)\to \R$. Note that
	\[||V^\e||(A) = \int_{A \cap \{\nabla u_\e \neq 0\}} \e\frac{|\nabla u_\e|^2}{2}\,d\mathcal{H}^n, \quad \mbox{for all Borel sets} \quad A \subset M.\]
In particular we have $||V^\e||(M) \leq E_\e(u_\e)$. 

The convergence of solutions of \eqref{eq:ac} to minimal hypersurface is described in terms of varifolds by the following

\begin{teo}[\cite{HutchinsonTonegawa, Tonegawa, TonegawaWickramasekera, Guaraco}] \label{thm:conv1}
Let $M^n$ be closed Riemannian manifold of dimension $n\geq 3$ and let $\{u_{\e_k}\}$ be a sequence of solutions to \eqref{eq:ac} in M with $\e=\e_k \downarrow 0$. Assume that there are positive constants $c_0$ and $E_0$, and a nonnegative integer $p$ such that
	\begin{equation} \label{hyp:conv}
		\limsup_k \sup_M|u_{\e_k}| \leq c_0, \quad \limsup_k E_{\e_k}(u_{\e_k}) \leq E_0 \quad \mbox{and} \quad \limsup_k m(u_{\e_k}) \leq p.
	\end{equation}
Then after perhaps passing to a subsequence the varifolds $V^k=V^{\e_k}$ converge to a $(n-1)$-rectifiable varifold $V$ in $M$, and we have
	\begin{enumerate}
		\item[(i)] $\sigma^{-1}V$ is a stationary integral varifold.
		\item[(ii)] $||V||(\phi) = \lim_k \int_M \e_k \frac{|\nabla u_{\e_k}|^2}{2}\phi \,d\mathcal{H}^n = \lim_k \int_M \frac{W(u_{\e_k})}{\e_k}\phi\,d\mathcal{H}^n.$
		\item[(iii)] The support $\Gamma$ of $V$ is an embedded minimal hypersurface outside of a singular $\sing V \subset \Gamma$ set of Hausdorff dimension at most $n-8$.
	\end{enumerate}
\end{teo}

As usual we will denote by $\reg V = \Gamma \setminus \sing V$ the regular part of $V$, that is, the set of points $x \in \Gamma$ such that $U \cap \Gamma$ is smoothly embedded for some neighborhood $U$ of $x$. Note that the item (ii) above shows that the equipartition of energy holds. Moreover from the constancy theorem for stationary varifolds we conclude that there are positive integers $m_1,\ldots, m_N$ such that $V=\sum_{j=1}^N v(\Gamma_j,\sigma m_j)$, where $\Gamma_1,\ldots, \Gamma_N$ are the connected components of $\Gamma$ and $v(K,\theta)$ denotes the rectifiable varifold induced by a $(n-1)$-rectifiable set $K \subset M$ with multiplicity $\theta$. In particular
	\[\lim_k E_{\e_k}(u_{\e_k}) = 2||V||(M) = 2\sigma \sum_{j=1}^N m_j \mathcal{H}^{n-1}(\Gamma_j).\]

\begin{obs} As noted by Hiesmayr in \cite{Fritz}, despite this definition of $V^\e$ differs slightly from the one given by Hutchinson and Tonegawa, the equipartition of energy shows that these definitions give the same limit varifold $V$.
\end{obs}

Using the convergence of the associated varifolds $V^{\e_k}$ we may obtain the following result. We will denote hereafter $\nu_\e=\nabla u_\e/|\nabla u_\e|$, which is well defined on the full $\mathcal{H}^n$-measure open subset of $M$ where $\nabla u_\e$ does not vanish.

\begin{prop} \label{prop1}
Let $\{u_{\e_k}\}$ be a sequence of solutions to \eqref{eq:ac} satisfying the hypothesis \eqref{hyp:conv}. After possibly passing to a subsequence it holds
	\[\int_{M} \e_k T(\nabla u_{\e_k}, \nabla u_{\e_k})\, d\mathcal{H}^n \to 2\sigma\,\sum_{j=1}^N m_j\int_{\Gamma_j} T(\vec n_j, \vec n_j) \, d\mathcal{H}^{n-1}\]
for any $(0,2)$-tensor $T$ on $M$. Here $\vec n_j$ denotes a measurable choice of an unit normal vector field defined a.e. on $\Gamma$.
\end{prop}

\begin{proof}
Define $\phi_T:G_{n-1}(M) \to \R$ by
	\[\phi_T(x,S) = \tr_g(T|_x) - \tr_g(S^*T|_x),\]
for $x \in M$ and $S \in G_{n-1}(T_xM)$, where we identify the subspace $S$ with the corresponding orthogonal projection $T_xM \to S$. The function $\phi_T$ is smooth and hence Theorem \ref{thm:conv1} gives (after possibly passing to a subsequence) $V^{\e_k}(\phi_T) \to V(\phi_T)$. On the other hand for every $\e=\e_k$ we have
	\[V^\e(\phi_T) = \int_{M} \frac{\e|\nabla u_\e|_x|^2}{2} \phi_T(x,T_x\{u_\e=u_\e(x)\})\,d\mathcal{H}^n(x),\]
and for every $x \in \{\nabla u_\e \neq 0\}$ if we pick an orthonormal basis $\{e_1, e_2,\ldots, e_n=\nu_\e\}$ of $T_xM$ then
	\[\phi_T(x,T_x\{u_\e=u_\e(x)\}) \! = \! \sum_{i=1}^n \left( T(e_i,e_i) - T(e_i-\left\langle e_i,\nu_\e\right\rangle\nu_\e, e_i-\left\langle e_i,\nu_\e\right\rangle\nu_\e) \right) \! = \! T(\nu_\e,\nu_\e).\]
Hence
	\[V^\e(\phi_T) = \int_{M \cap \{\nabla u_\e \neq 0\}} \frac{\e|\nabla u_\e|^2}{2} T(\nu_\e,\nu_\e)\,d\mathcal{H}^n = \frac{\e}{2} \int_M T(\nabla u_\e,\nabla u_\e) \, d\mathcal{H}^{n}.\]
Similarly, we have
	\[V(\phi_T) = \sigma \sum_{j=1}^N m_j \int_{\Gamma_j} \phi_T(x, T_x\Gamma_j) \, d\mathcal{H}^{n-1}(x) = \sigma \sum_{j=1}^N m_j \int_{\Gamma_j} T(\vec n_j, \vec n_j) \, d\mathcal{H}^{n-1}(x).\]
This concludes the proof.
\end{proof}

\section{A second inner variation formula in Riemannian manifolds} \label{sec:2var}

In this section we re-derive the second inner variation formula for the Allen-Cahn energy $E_\e$ obtained by Le, see \cite[(2.2)]{LeInner}. 

We recall the definition of \emph{inner variations}. Given a smooth vector field $X$ on $M$ we denote by $\Phi^t$ its time $t$ flow, for sufficiently small $t$. For $u \in H^1(M)$ we write $u^t := u \circ (\Phi^t)^{-1} = u \circ \Phi^{-t}$. The \emph{first inner variation}, respectively the \emph{second inner variation}, of $E_\e$ at $u$ with respect to the vector field $X$ are defined by
	\[\delta E_\e(u,X) = \dfrac{d}{dt}\bigg|_{t=0} E_\e(u^t) \quad \mbox{and} \quad \delta^2 E_\e(u,X) = \frac{d^2}{dt^2}\bigg|_{t=0} E_\e(u^t),\]
respectively. Observe that
	\[\frac{d}{dt}\bigg|_{t=0} u^t(x) = du_x(-X|_x) = -\left\langle\nabla u,X\right\rangle_x \quad \mbox{for all} \quad x \in M,\]
hence we may express $\delta E_\e(u,X)$ in terms of the first derivative of $E_\e$ as
	\[ \delta E_\e(u,X) = E_\e'(u)(-\left\langle\nabla u,X\right\rangle).\]
Moreover, if $u$ is a critical point of $E_\e$, then we have also
	\begin{equation} \label{eq:2var} \delta^2 E_\e(u,X) = E_\e''(u)\left(\left\langle\nabla u,X\right\rangle,\left\langle\nabla u,X\right\rangle\right).
	\end{equation}
In light of the analogy between solutions of the Allen-Cahn equation and minimal hypersurfaces, inner variations may be regarded as \emph{geometric} variations of the energy functional in contrast with \emph{analytic} variations by arbitrary functions in $H^1(M)$. By that we mean these inner variations arise in a similar fashion of variations of the area functional, and hence it is reasonable to expect that variations which decrease area of limit interfaces also decrease $E_\e$ for sufficiently small $\e$. The second inner variation formula shows that this is the case, provided the corresponding vector field satisfies some additional conditions.

Using the change of variables formula for $y=\Phi^t(x)$ one writes
	\begin{align*}
		E_\e(u^t) & = \int_M \left( \frac{\e \left|\nabla u^t|_y\right|^2}{2} + \frac{W(u^t(y))}{\e} \right)\, d\mathcal{H}^{n}(y)\\
		& = \int_M \left( \frac{\e\left|\nabla(u \circ \Phi^{-t})|_{\Phi^t(x)}\right|^2}{2} + \frac{W(u(x))}{\e} \right) |J\Phi^t(x)|\,d\mathcal{H}^n(x),
	\end{align*}
where $|J\Phi^t|$ is the Jacobian determinant of $\Phi^t$. If $\{e_j^t\}_{j=1}^n$ is an orthonormal basis for $T_{\Phi^t(x)}M$ then it holds
	\begin{align*}
	\left|\nabla (u \circ \Phi^{-t})|_{\Phi^t(x)}\right|^2 & = \sum_{j=1}^n \left(\left(du|_x \circ d\Phi^{-t}|_{\Phi^t(x)}\right)(e_j^t) \right)^2\\
		& = \sum_{j=1}^n du_x(v_j^t)^2 = \sum_{j=1}^n g^t\left( \nabla^{g^t}u|_x, v_j^t\right)^2,
	\end{align*}
where $v_j^t = d\Phi^{-t}|_{\Phi^t(x)}e_i^t$ and $g^t$ is the pullback metric of $g$ by $\Phi^t$, that is $g^t=(\Phi^t)^*g$. Since $g^t(v_j^t,v_k^t) = g(e_i^t,e_j^t) = \delta_{jk}$ we get
	\[\left|\nabla (u \circ \Phi^{-t})|_{\Phi^t(x)}\right|^2 = \left|\nabla^{g^t} u|_x\right|_{g_t}^2 = g^t_{ij}(g^t)^{ia} (g^t)^{jb}(\partial_a u)(\partial_b u),\]
where we sum over repeated indexes and use local coordinates.

The next lemma gives the first and second variation formulas for the metric $g^t$ and its inverse. Here $R$ is the $(0,4)$ curvature tensor of $(M,g)$.

\begin{lema} \label{lem:2var}
Let $h^t_{ij} := \frac{d}{dt}g^t_{ij}$. In any coordinate system we have
	\begin{enumerate}
		\item[(i)] $h_{ij}^0 = \frac{d}{dt}\big|_{t=0}g^t_{ij} = \left\langle\nabla_{\partial_i}X,\partial_j\right\rangle + \left\langle\partial_i, \nabla_{\partial_j}X\right\rangle,$
		\item[(ii)] $\frac{d}{dt}\big|_{t=0}(g^t)^{ia} = -g^{ik}h_{h\ell}g^{\ell a},$
		\item[(iii)] $\resizebox{0.9\hsize}{!}{$\frac{d}{dt}\big|_{t=0}h^t_{ij} \! = \! \left\langle\nabla_{\partial_i}\nabla_XX, \partial_j\right\rangle + \left\langle\partial_i, \nabla_{\partial_j}\nabla_{X}X\right\rangle + 2\left\langle\nabla_{\partial_i}X,\nabla_{\partial_j}X\right\rangle-2R(X,\partial_i,X,\partial_j)$}$,
		\item[(iv)] $\frac{d^2}{dt^2}\big|_{t=0}(g^t)^{ia} = 2g^{ir}h_{rs}g^{sk}h_{k\ell}g^{\ell a}-g^{ik}g^{\ell a}\left(\frac{d}{dt}\big|_{t=0}h^t_{k\ell}\right).$
	\end{enumerate}
\end{lema}

Using Lemma \ref{lem:2var} we may determine the derivatives of $|\nabla u^t|^2$. It holds
	\begin{align*}
		\frac{d}{dt}\bigg|_{t=0}\left|\nabla (u \circ \Phi^{-t})|_{\Phi^t(x)}\right|^2 & = h(\nabla u, \nabla u) -2 g_{ij}g^{ik}h_{k\ell}(g^{\ell a}\partial_a u)(g^{jb}\partial_bu) \\
		& =-h(\nabla u,\nabla u) = -2\left\langle\nabla_{\nabla u}X,\nabla u\right\rangle
	\end{align*}
and
	\begin{align*}
		\frac{d^2}{dt^2}&\bigg|_{t=0}\left|\nabla (u \circ \Phi^{-t})|_{\Phi^t(x)}\right|^2\\
		&= \left[ -2R(X,\nabla u,X,\nabla u) + 2\left\langle\nabla_{\nabla u}\nabla_XX, \nabla u\right\rangle + 2|\nabla_{\nabla u}X|^2 \right]\\
		& + 2\left[2g^{sk}h_{ik}h_{js}(g^{ia}\partial_a u)(g^{jb}\partial_b u) + 2R(X,\nabla u,X,\nabla u)\right. \\
		&\qquad \qquad \qquad \quad \left. - 2\left\langle\nabla_{\nabla u}\nabla_XX, \nabla u\right\rangle - 2|\nabla_{\nabla u}X|^2 \right]-2g^{ks}h_{ik}h_{js}(g^{ia}\partial_a u)(g^{jb}\partial_b u)\\
		& = 2R(X,\nabla u,X,\nabla u) -2\left\langle\nabla_{\nabla u}\nabla_XX, \nabla u\right\rangle + 2|du \circ \nabla X|^2 + 4 \left\langle\nabla_{\nabla_{\nabla u}X}X, \nabla u\right\rangle.
	\end{align*}
Moreover, we can obtain the derivatives of $|J\Phi^t|$ as the derivatives of the determinant of the metric $g^t$. More precisely we have
	\[\frac{d}{dt}\bigg|_{t=0}|J\Phi^t| = \tr_g h = 2\Div X\]
and
	\begin{align*}
		\frac{d^2}{dt^2}\bigg|_{t=0}|J\Phi^t| & = \frac{1}{2}\left[\tr_g\left(\tfrac{d}{dt}\big|_{t=0}h^t\right) - \tr_g\left((h_{ij})^2 \right) + \tfrac{1}{2}(\tr_g h)^2 \right]\\
		& = \frac{1}{2}\left[ 2\Div(\nabla_XX)-2\Ric(X,X) +2 \tr_g S_X -|h|^2 + 2(\Div X)^2 \right],
	\end{align*}
where $S_X$ is the $(0,2)$-tensor $S_X(Y_1,Y_2) = \left\langle\nabla_{Y_1}X,\nabla_{Y_2}X\right\rangle$. 

From the calculations above one concludes

\begin{prop}
It holds
	\[\delta E_\e(u,X) = \int_M\left[\left(\frac{\e|\nabla u|^2}{2}+\frac{W(u)}{\e}\right) \Div X - \e \left\langle\nabla_{\nabla u}X,\nabla u\right\rangle\right]\,d\mathcal{H}^n\]
and
	\begin{align} \label{eq:secv}
		\delta^2 E_\e(u,X) & = \int_M \left\{ \Div(\nabla_XX)-\Ric(X,X) + \tr_gS_X - \frac{1}{2}|h_X|^2 +(\Div X)^2 \right\}\,de_\e  \nonumber \\
		& + \e \int_M \left\{T_X(\nabla u,\nabla u) + 2\left\langle\nabla_{\nabla_{\nabla u}X} X, \nabla u\right\rangle - \left\langle\nabla_{\nabla u}\nabla_XX, \nabla u\right\rangle \right.\\
		& \qquad \qquad \qquad \quad \left.- 2\left\langle\nabla_{\nabla u}X,\nabla u\right\rangle\Div X + R(X,\nabla u,X, \nabla u) \right\} \, d\mathcal{H}^n, \nonumber
	\end{align}
where
	\begin{align*}
		S_X(Y_1,Y_2) &= \left\langle \nabla_{Y_1}X,\nabla_{Y_2}X\right\rangle,\\
		h_X(Y_1,Y_2) &= \left\langle \nabla_{Y_1} X, Y_2\right\rangle + \left\langle Y_1,\nabla_{Y_2} X\right\rangle,\\
		T_X(Y_1,Y_2) &= \tr_g \left((Z_1,Z_2) \mapsto \left\langle \nabla_{Z_1}X,Y_1\right\rangle \left\langle\nabla_{Z_2}X,Y_2\right\rangle \right).
	\end{align*}
\end{prop}

We are now in position to describe the limit of $\delta^2 E_\e(u_\e,X)$ as $\e \downarrow 0$, for solutions $\{u_\e\}$ of \eqref{eq:ac} with uniformly bounded Morse index and energy. Under these conditions, it follows from Theorem \ref{thm:conv1} that, after perhaps passing to a subsequence,
	\begin{align*}
		&\int_M\left\{ \Div(\nabla_XX)-\Ric(X,X) + \tr_gS_X - \frac{1}{2}|h_X|^2 +(\Div X)^2 \right\} \,de_\e \to\\
		&\quad \to  2\sigma \sum_{j=1}^N m_j \int_{\Gamma_j} \left\{ \Div(\nabla_XX) - \Ric(X,X) + \tr_gS_X - \frac{1}{2}|h_X|^2 +(\Div X)^2 \right\} \, d\mathcal{H}^{n-1}
	\end{align*}
as $\e \downarrow 0$. We can use Proposition \ref{prop1} to verify that the second term in the RHS of \eqref{eq:secv} converges to
	\begin{align*}
		2\sigma\sum_{j=1}^Nm_j &\int_{\Gamma_j} \bigg\{T_X(\vec n_j,\vec n_j) + 2\left\langle\nabla_{\nabla_{\vec n_j}X} X, \vec n_j\right\rangle - \left\langle\nabla_{\vec n_j}\nabla_XX, \vec n_j\right\rangle \\
		& \qquad \qquad \qquad  \qquad - 2\left\langle\nabla_{\vec n_j}X,\vec n_j\right\rangle\Div X + R(X,\vec n_j,X, \vec n_j) \bigg\} \, d\mathcal{H}^n.
	\end{align*}
If we write $\Div Y = \Div^{\Gamma_j} Y + \left\langle\nabla_{\vec n_j} Y, \vec n_j\right\rangle$ for $Y=X$ and $Y=\nabla_XX$, and recall that $V$ is stationary, then we can write $\lim_\e\delta^2E_\e(u_\e,X)$ as
	\[ \sum_{j=1}^n m_j \int_{\Gamma_j} \left\{-\Ric(X,X) + (\Div^{\Gamma_j} X)^2 + R(X,\vec n_j,X, \vec n_j) + A_j + B_j \right\}\,d\mathcal{H}^{n-1},  \]
where, choosing an orthonormal basis $\{e_j\}_{j=1}^n$ with $e_n =\vec n_j$ for $T_xM$ and $x \in \Gamma_j\cap \reg V$,
	\[A_j = T_X(\vec n_j,\vec n_j) -\left\langle \nabla_{\vec n_j}X,\vec n_j\right\rangle^2 = \sum_{\ell=1}^{n-1} \left\langle\nabla_{e_\ell}X,\vec n_j\right\rangle^2 = |\nabla^\perp X|^2\]
and
	\begin{align*}
		B_j & = 2\left\langle\nabla_{\nabla_{\vec n_j}X}X,\vec n_j\right\rangle + \tr_g S_X - |h_X|^2/2 \\
		& = 2\sum_{\ell=1}^{n-1} \left\langle\nabla_{e_\ell}X, \vec n_j\right\rangle \left\langle\nabla_{\vec n_j}X, e_\ell\right\rangle + 2\left\langle \nabla_{\vec n_j}X, \vec n_j\right\rangle^2 + \sum_{\ell=1}^n |\nabla_{e\ell} X|^2 \\
		& \qquad - \frac{1}{2} \sum_{\ell,m=1}^n\left( \left\langle\nabla_{e_\ell} X, e_m\right\rangle^2 + 2\left\langle \nabla_{e_\ell}X,e_m\right\rangle \left\langle\nabla_{e_m}X,e_\ell\right\rangle + \left\langle \nabla_{e_m}X,e_\ell\right\rangle^2 \right)\\
		& = \left\langle\nabla_{\vec n_j}X,\vec n_j\right\rangle^2 - \sum_{\ell,m=1}^{n-1} \left\langle\nabla_{e_\ell}X,e_m\right\rangle \left\langle \nabla_{e_m}X, e_\ell \right\rangle.
	\end{align*}
Therefore (see \cite[9.4]{Simon}) we get

\begin{prop} \label{prop2}
Assume $\{u_{\e_k}\}$ is a sequence of solutions of \eqref{eq:ac} which satisfy the uniform bounds \eqref{hyp:conv}. Then there is a (not relabeled) subsequence of $\{u_{\e_k}\}$ for which all of the conclusions of Theorem \ref{thm:conv1} hold, and moreover
	\begin{equation} \label{eq:limitvar}
		\resizebox{0.93\hsize}{!}{$\displaystyle\frac{1}{2\sigma}\lim_k\delta^2 E_{\e_k}(u_{\e_k},X) = \delta^2V(X) + \sum_{j=1}^N m_j \int_{\Gamma_j} \left\{ \left\langle \nabla_{\vec n_j}X,\vec n_j \right\rangle^2 + R(X,\vec n_j,X,\vec n_j) \right\}\,d\mathcal{H}^{n-1}$}
	\end{equation}
for all smooth vector fields $X$ defined on $M$.
\end{prop}

Similarly to the Euclidean case, the error term $\lim_\e \delta^2E_\e(u_\e,X)-2\sigma\delta^2V(X)$ is positive provided $X$ is orthogonal to $\Gamma$. We also point out that the formula above holds even when the limit interface is one-sided, and the RHS is defined in terms of local smooth (or merely measurable) choices for $\vec n_j$, since the integrand in the error does not change if we replace $\vec n_j$ by $-\vec n_j$. Finally note that this extra term vanishes whenever $X$ is an extension of a given normal field defined on $\Gamma$ and it satisfies $\left\langle\nabla_{\vec n_j}X,\vec n_j\right\rangle \equiv 0$.

\section{Proof of Theorem \ref{thm:a}} \label{sec:thma}

The proof follows from the arguments given in the proof of \cite[Corollary 1.1]{LeEigen} replacing the corresponding convergence result by Proposition \ref{prop2}, and from the remarks of \cite[\S 3.2]{Fritz}. We include it here for the reader's convenience.

We fix $U \subset \! \subset M \setminus \sing V$. Denote by $Q_\e$ the quadratic form given by the second inner variation of $E_{\e}$, namely
	\[Q_\e(X) = \delta^2E_{\e}(u_{\e},X) = E_\e''(u)\left(\left\langle\nabla u,X\right\rangle,\left\langle\nabla u,X\right\rangle\right)\]
for $H^1$ vector fields $X$ in $M$ (this is well defined due to \eqref{eq:2var}). When $\e=\e_k$ we will write simply $Q_{\e_k}=Q_k$. We consider also the quadratic form $Q_V$ given by the second variation of $V$, that is
	\[Q_V(X) = \delta^2V(X) = \sum_{j=1}^N m_j\int_{\Gamma_j} |\nabla^\perp X|^2 - (\Ric(X,X) + |A_{\Gamma_j}|^2|X|^2)\]
for $H^1_0$ normal vector fields $X$ defined on $\reg V$. By the extension result of the \hyperref[app:ext]{Appendix} all compactly supported smooth normal vector fields $X$ defined on $U \cap \reg V$ admit a compactly supported smooth extension $\tilde X$ to $U$ such that $\left\langle\nabla_{\vec n_j} \tilde X, \vec n_j\right\rangle$ vanishes on $U \cap \Gamma_j$ for all $j$. For such extensions both terms in the integrand in the RHS of \eqref{eq:limitvar} vanish and hence
	\begin{equation} \label{eq:quad}
	 \lim_k Q_k(\tilde X) = 2\sigma\,Q_V(X)
	\end{equation}
whenever $\{u_{\e_k}\}$ satisfies the hypothesis \eqref{hyp:conv}. 

Recall that the eigenvalues $\lambda_\ell^\e(U)$ of the operator $-\e\Delta+W''(u_\e)/\e$ on $U$ have the following variational characterization:
	\[\lambda_\ell^\e(U) = \inf_{\dim S=\ell} \max_{\phi \in S\setminus 0} \frac{E_\e(u_\e)''(\phi,\phi)}{||\phi||^2_{L^2(M)}},\]
where $S$ varies among linear subspaces of $H^1_0(U)$. Similarly the eigenvalues of the Jacobi operator of $U \cap \reg V$ satisfy
	\[\lambda_\ell(U) = \inf_{\dim E=\ell} \max_{X \in E\setminus 0} \frac{\sum_{j=1}^n \int_{\Gamma_j} \left\{|\nabla^\perp X|^2 - (\Ric(X,X)+|A_{\Gamma_j}|^2|X|^2)\right\}}{||X||^2_{L^2(\reg V)}}\]
where $E$ varies among linear subspaces of $H_0^1$ normal vector fields on $U \cap \reg V$. As noted by \cite[\S 3.2]{Fritz} in the case of the scalar second variation of the area, we can describe $\lambda_\ell(U)$ in terms of $Q_V$ by the \emph{weighted} min-max characterization
	\begin{equation} \label{eq:weight}
		\lambda_\ell(U) = \inf_{\dim E=\ell} \max_{X \in E\setminus 0} \frac{Q_V(X)}{||X||^2_{L^2(V)}},
	\end{equation}
where
	\[||X||^2_{L^2(V)} = \int_M |X|^2 \, d||V|| = \sum_{j=1}^N m_j \int_{\Gamma_j}|X|^2\, d\mathcal{H}^{n-1}.\]
In fact given linearly independent $H_0^1$ normal vector fields $X_1,\ldots, X_\ell$ on $U \cap \reg V$ let
	\[\hat X_i := m_j^{-1/2} X_i \quad \mbox{on} \quad \Gamma_j \cap (U\cap \reg V), \quad \mbox{for} \quad i=1,\ldots, \ell.\]
and $\hat X_i = 0$ on $U \cap (\Gamma_j \setminus \reg V)$. Then $\hat X_i$ are also $H_0^1$ normal vector fields on $U \cap \reg V$,
	\[||\hat X_i||_{L^2(V)}^2 = \sum_{j=1}^N\int_{\Gamma_j}|X_i|^2\,d\mathcal{H}^{n-1} = ||X_i||^2_{L^2(\reg V)},\]
	\[Q_V(\hat X_i) = \sum_{j=1}^N\int_{\Gamma_j}\left\{|\nabla^\perp X_i|^2 - (\Ric(X_i,X_i)+|A_{\Gamma_j}|^2|X_i|^2)\right\}\,d\mathcal{H}^{n-1}\]
and $\hat X_1,\ldots, \hat X_\ell$ are linearly independent. This proves \eqref{eq:weight}.

Given $\delta>0$, there is an $\ell$-dimensional linear space $E$ of smooth and compactly supported normal vector fields on $U\cap \reg V$, spanned by say $X_1,\ldots, X_\ell$, such that
	\[\max_{a \in S^{\ell-1}} \frac{Q_V(a \cdot X)}{||a\cdot X||_{L^2(V)}^2} \leq \lambda_\ell(U) + \delta,\]
where we use the notation $a \cdot X = \sum_{i=1}^\ell a_i X_i$ for $a \in \R^\ell$, and similarly for the corresponding extensions $\tilde X_1,\ldots, \tilde X_\ell$ on $U$. The map
	\[a \cdot X \in E \mapsto \left\langle \nabla u_{\e_k}, -a\cdot \tilde X \right\rangle \in H_0^1(U)\]
is injective for sufficiently large $k$, otherwise Proposition \ref{prop1} would give us
	\[\sum_{j=1}^N m_j \int_{\Gamma_j} \left\langle\vec n_j,a\cdot \tilde X\right\rangle^2 \,d\mathcal{H}^{n-1}=0\]
for some $a \in S^{\ell-1}$, which imply $a \cdot X=0$, a contradiction. Hence, the space of all $\langle\nabla u_{\e_k},-a\cdot\tilde X\rangle$ for $a \in \R^\ell$ is an $\ell$-dimensional linear subspace of $H_0^1(U)$ and
	\[\frac{\lambda_\ell^{\e_k}(U)}{\e_k} \leq \max_{a \in S^{\ell-1}} \frac{Q_{k}(a \cdot \tilde X)}{\e\int_M \left\langle \nabla u_{\e_k}, a \cdot \tilde X \right\rangle^2}\]
for large $k$. Choose a (not relabeled) subsequence of $\e_k \downarrow 0$ such that
	\[\lim_k\frac{\lambda_\ell^{\e_k}(U)}{\e_k} = \limsup_\e \frac{\lambda_\ell^\e(U)}{\e} =: \mu_\ell(U).\]
Using Proposition \ref{prop1}, \eqref{eq:quad} and the polarization formula for the quadratic forms $Q_k$ and
	\[q_k(Z):= \e_k \int_M \left\langle \nabla u_{\e_k},Z\right\rangle^2\,d\mathcal{H}^n,\]
 we conclude that, after perhaps passing again to a subsequence, there exists $a \in S^{\ell-1}$ such that 
	\[\mu_\ell(U) -\delta \leq \frac{Q_{k}(a \cdot \tilde X)}{\e_k\int_M \left\langle \nabla u_{\e_k},a\cdot \tilde X\right\rangle^2}\]
for sufficiently large $k$. On the other hand $Q_{k}(a \cdot \tilde X) \to 2\sigma Q_V(a \cdot X)$ and
	\[\e_k\int_M \left\langle\nabla u_{\e_k}, a \cdot \tilde X \right\rangle^2 \to 2\sigma \sum_{j=1}^N m_j \int_{\Gamma_j} \left\langle \vec n_j, a\cdot \tilde X \right\rangle^2\,d\mathcal{H}^{n-1} = 2\sigma||a \cdot X||_{L^2(V)}^2,\]
where we used again Proposition \ref{prop1}. Thus
	\[\mu_\ell(U) -\delta \leq \frac{Q_V(a \cdot X)}{||a \cdot X||_{L^2(V)}^2} \leq \lambda_\ell(U)+\delta.\]
Since $\delta>0$ is arbitrary, this proves that $\mu_\ell(U) \leq \lambda_\ell(U)$. To verify the last statement of Theorem \eqref{thm:a}, recall that the Morse index of the regular part of $\Gamma$, viz. $\reg V$, is given by
	\[\mathrm{ind}(\reg V) = \sup_U\# \{\ell \in \N : \lambda_\ell(U)<0\} \]
for $U \subset \! \subset M \setminus \sing V$ such that $U$ intersects $\reg V$. It follows then from the spectral upper bound $\mu_\ell(U) \leq \lambda_\ell(U)$ and the Morse index bound $\limsup_k m(u_{\e_k}) \leq p$ that $\mathrm{ind}(\reg V) \leq p$. This finishes the proof.
\qed
		
\begin{appendix}
\section*{Appendix}
\label{app:ext}
Let $M$ be a complete Riemannian manifold and let $\Sigma \subset M$ be a hypersurface. Given $\delta>0$ denote
	\[B_\delta(N\Sigma) = \{ v \in TM|_{\Sigma} : v \perp T\Sigma, |v|<\delta\},\]
and let $\pi:B_\delta(N\Sigma) \to \Sigma$ be the projection map. Assume we can find $\delta>0$ such that the \emph{normal exponential map} $F:B_\delta(N\Sigma) \to M$ given by $F(v_x)=\exp_xv$ is a diffeomorphism onto an open set $U_\delta \subset M$. Under these conditions, we have

\begin{lema0}
Every compactly supported smooth normal vector field $X$ defined on $\Sigma$ admits a smooth extension $\tilde X$ to $M$ such that $\nabla_v \tilde X=0$ for all $v \in N\Sigma \subset TM$. Moreover we can assume that $\tilde X$ vanishes outside a neighborhood $U \subset U_\delta$ of $\Sigma$.
\end{lema0}

\begin{proof}
Intuitively we will define $\tilde X$ at some $x \in U_\delta$ via parallel transport along $\gamma:[0,r] \to M$, the unique geodesic which is normal to $\Sigma$ at $\gamma(0) \in \Sigma$ and satisfies $|\gamma'|=1$ and $\gamma(r)=x$. Then, we may use a cutoff function to extend $\tilde X$ to $M$.

To make this idea precise write $p=\pi \circ F^{-1}:U_\delta \to \Sigma$, and for every $x \in U_\delta$ let
	\[\hat X|_x = \dfrac{d}{dt}\bigg|_{t=0} \exp_{p(x)}(F^{-1}(x) + tX|_{p(x)}) = d\exp_{p(x)}(F^{-1}(x))\,(X|_{p(x)}).\]
We have that $\hat X$ is a vector field on $U_\delta$ which is as regular as $X$. One verifies that if $x \in \Sigma$ then $F^{-1}(x)=0 \in T_xM$ and $p(x)=x$, so $\hat X|_x = X|_x$ and $\hat X$ extends $X$ to $U_\delta$. Moreover since the fibers of $N\Sigma$ have dimension $1$, whenever $x \in U_\delta$ is such that $X|_{p(x)} \neq 0$ we have 
	\[F^{-1}(x)=s(x)X|_{p(x)} \quad \mbox{for some} \quad s(x) \in \R,\]
with $s(x)=0$ for $x \in \Sigma$. Hence $\hat X|_{x}$ is the tangent vector to the normal geodesic 
	\[\gamma(t)=\exp_{p(x)}(tX|_{p(x)})\]
at $t=s(x)$ and for all $x \in \Sigma$ in which $X$ does not vanish we have
	\[0 = \dfrac{D\gamma'}{dt}\bigg|_{t=0} = \nabla_{X|_x}\hat X.\]
This proves that
	\[\nabla_v \hat X|_x=0 \quad \mbox{for all} \quad x \in \Sigma \quad \mbox{s.t.} \quad X|_x \neq 0, \quad \mbox{and} \quad v \in N_x\Sigma.\]
Finally if $x \in \Sigma$ and $X|_x=0$ then $\hat X$ vanishes along the points in $p^{-1}(x) \cap U_\delta$, namely the geodesic $\gamma(t)=\exp_x(t \vec n)$ for an unit $\vec n \in N_x\Sigma$ and $|t|<\delta$. We conclude thus
	\[0 = \dfrac{D (\hat X|_\gamma)}{dt}\bigg|_{t=0} = \nabla_{\gamma'(0)}\hat X = \nabla_{\vec n} \hat X\]
and $\nabla_v \hat X$ vanishes for all $v \in N_x\Sigma$. 

To obtain the required extension $\tilde X$ on $M$ we simply use a cutoff function $\rho \in C^\infty(M)$ which vanishes outside a neighborhood $U_1 \subset U_\delta$ of $\Sigma$ and such that $\rho \equiv 1$ on a smaller neighborhood $U\subset\!\subset U_1$ of $\Sigma$, and let $\tilde X=\rho \hat X$. 

\end{proof}

\end{appendix}
	
\bibliography{main}

\begin{thebibliography}{10}

\bibitem{AllenCahn}
{\sc J.~Cahn and S.~Allen}, {\em A microscopic theory for domain wall motion
  and its experimental verification in {Fe}-{Al} alloy domain growth kinetics},
  Le Journal de Physique Colloques, 38 (1977), pp.~C7--51.

\bibitem{GasparGuaraco}
{\sc P.~Gaspar and M.~A. Guaraco}, {\em The {A}llen-{C}ahn equation on closed
  manifolds}, arXiv:1608.06575 [math.DG],  (2016).

\bibitem{Guaraco}
{\sc M.~A. Guaraco}, {\em Min-max for phase transitions and the existence of
  embedded minimal hypersurfaces}, arXiv:1505.06698 [math.DG],  (2015).

\bibitem{Fritz}
{\sc F.~Hiesmayr}, {\em Spectrum and index of two-sided {A}llen-{C}ahn minimal
  hypersurfaces}, arXiv:1704.07738 [math.DG],  (2017).

\bibitem{HutchinsonTonegawa}
{\sc J.~E. Hutchinson and Y.~Tonegawa}, {\em Convergence of phase interfaces in
  the van der {W}aals-{C}ahn-{H}illiard theory}, Calculus of Variations and
  Partial Differential Equations, 10 (2000), pp.~49--84.

\bibitem{Ilmanen}
{\sc T.~Ilmanen}, {\em Convergence of the {A}llen-{C}ahn equation to {B}rakke's
  motion by mean curvature}, J. Differential Geom, 38 (1993), pp.~417--461.

\bibitem{KohnSternberg}
{\sc R.~V. Kohn and P.~Sternberg}, {\em Local minimisers and singular
  perturbations}, Proceedings of the Royal Society of Edinburgh Section A:
  Mathematics, 111 (1989), pp.~69--84.

\bibitem{LeInner}
{\sc N.~Q. Le}, {\em On the second inner variation of the {A}llen-{C}ahn
  functional and its applications}, Indiana University Mathematics Journal,
  (2011), pp.~1843--1856.

\bibitem{LeEigen}
\leavevmode\vrule height 2pt depth -1.6pt width 23pt, {\em On the second inner
  variations of {A}llen--{C}ahn type energies and applications to local
  minimizers}, Journal de Math{\'e}matiques Pures et Appliqu{\'e}es, 103
  (2015), pp.~1317--1345.

\bibitem{Mantoulidis1}
{\sc C.~Mantoulidis}, {\em Allen-{C}ahn min-max on surfaces}, arXiv:1706.05946
  [math.AP],  (2017).

\bibitem{MarquesNevesIndex}
{\sc F.~C. Marques and A.~Neves}, {\em Morse index and multiplicity of min-max
  minimal hypersurfaces}, Cambridge Journal of Mathematics, 4 (2016),
  pp.~463--511.

\bibitem{MarquesNevesInfinite}
{\sc F.~C. Marques and A.~Neves}, {\em Existence of infinitely many minimal
  hypersurfaces in positive {R}icci curvature}, Invent. Math., 209 (2017),
  pp.~577--616.

\bibitem{Modica}
{\sc L.~Modica}, {\em The gradient theory of phase transitions and the minimal
  interface criterion}, Archive for Rational Mechanics and Analysis, 98 (1987),
  pp.~123--142.

\bibitem{ModicaMortola}
{\sc L.~Modica and S.~Mortola}, {\em Il limite nella {$\Gamma $}-convergenza di
  una famiglia di funzionali ellittici}, Boll. Un. Mat. Ital. A (5), 14 (1977),
  pp.~526--529.

\bibitem{Pacard}
{\sc F.~Pacard}, {\em The role of minimal surfaces in the study of the
  {A}llen-{C}ahn equation}, in Geometric analysis: partial differential
  equations and surfaces, vol.~570 of Contemp. Math., Amer. Math. Soc.,
  Providence, RI, 2012, pp.~137--163.

\bibitem{Reshetnyak}
{\sc Y.~G. Reshetnyak}, {\em Weak convergence of completely additive vector
  functions on a set}, Siberian Mathematical Journal, 9 (1968), pp.~1039--1045.

\bibitem{Savin}
{\sc O.~Savin}, {\em Phase transitions, minimal surfaces and a conjecture of
  {D}e {G}iorgi}, Current developments in mathematics,  (2009), pp.~59--113.

\bibitem{SchoenSimon}
{\sc R.~Schoen and L.~Simon}, {\em Regularity of stable minimal hypersurfaces},
  Communications on Pure and Applied Mathematics, 34 (1981), pp.~741--797.

\bibitem{Simon}
{\sc L.~Simon}, {\em Lectures on geometric measure theory}, The Australian
  National University, Mathematical Sciences Institute, Centre for Mathematics
  \& its Applications, 1983.

\bibitem{Smith}
{\sc G.~Smith}, {\em Bifurcation of solutions to the {A}llen--{C}ahn equation},
  Journal of the London Mathematical Society, 94 (2016), pp.~667--687.

\bibitem{Sternberg}
{\sc P.~Sternberg}, {\em The effect of a singular perturbation on nonconvex
  variational problems}, Archive for Rational Mechanics and Analysis, 101
  (1988), pp.~209--260.

\bibitem{Tonegawa}
{\sc Y.~Tonegawa}, {\em On stable critical points for a singular perturbation
  problem}, Communications in Analysis and Geometry, 13 (2005), pp.~439--459.

\bibitem{TonegawaWickramasekera}
{\sc Y.~Tonegawa and N.~Wickramasekera}, {\em Stable phase interfaces in the
  van der {W}aals--{C}ahn--{H}illiard theory}, Journal f{\"u}r die reine und
  angewandte Mathematik (Crelles Journal), 2012 (2012), pp.~191--210.

\bibitem{WangWei}
{\sc K.~Wang and J.~Wei}, {\em Finite morse index implies finite ends},
  arXiv:1705.06831 [math.AP],  (2017).

\bibitem{Wickramasekera}
{\sc N.~Wickramasekera}, {\em A general regularity theory for stable
  codimension 1 integral varifolds}, Annals of Mathematics, 179 (2014),
  pp.~843--1007.

\end{thebibliography}
\bibliographystyle{siam}
	
\end{document}